\documentclass[12pt, reqno]{amsart}
\usepackage[margin=1in]{geometry}

\usepackage{amsmath,amssymb,amsthm,mathtools}
\usepackage{enumitem}
\usepackage{charter}
\usepackage{newtxmath}

\allowdisplaybreaks

\usepackage{xcolor} %
\definecolor{darkred}{rgb}{0.5,0,0} %
\definecolor{darkblue}{rgb}{0,0,0.5} %

\usepackage[colorlinks,linkcolor=darkred,citecolor=darkblue,urlcolor=darkblue,hypertexnames=true]{hyperref}

\usepackage{euflag}

\newtheorem{theorem}{Theorem}[section]

\newtheorem{corollary}[theorem]{Corollary}
\newtheorem{lemma}[theorem]{Lemma}

\newtheorem{proposition}[theorem]{Proposition}
\newtheorem{prop}[theorem]{Proposition}
\newtheorem*{theorem*}{Theorem}
\newtheorem{customthm}{Theorem}

\theoremstyle{definition}
\newtheorem{definition}{Definition}[section]
\theoremstyle{definition}
\newtheorem{remark}{Remark}[section]

\newtheorem{example}[theorem]{Example}

\newcommand{\vd}{{\tt{d}}}

\numberwithin{equation}{section}

\title[UCP Maps on Finite Toeplitz Systems]
{UCP Maps on Finite Toeplitz Systems}

\author[R.\ Duhan]{Ritul Duhan}
\address{Ritul Duhan, Department of Mathematics, Indian Institute of Science, Bengaluru, India}
\email{ritul2023@iisc.ac.in}
\thanks{RD was supported by the Prime Minister's Research Fellowship (PMRF ID 0202985) of the Government of India, and the DST FIST program 2021 (TPN-700661).}

\author[A.\ Jindal]{Abhay Jindal}
\address{Abhay Jindal, Faculty of Mathematics and Physics, University of Ljubljana, Slovenia}
\email{abhay.jindal@fmf.uni-lj.si}
 \thanks{AJ was supported by the Slovenian Research Agency program P1-0222 and
grant J1-50002. This work started when AJ was a research associate at the Indian
Institute of Science, Bengaluru, where he was supported by the JC Bose Fellowship
JCB/2021/000041 of ANRF. AJ thanks \'Ecole Polytechnique and Inria for their
hospitality during the preparation of this manuscript. This work was performed
within the project COMPUTE, funded within the QuantERA II Programme, which has
received funding from the European Union's Horizon 2020 research and innovation
programme under Grant Agreement No.~101017733 {\normalsize\protect\euflag}.}

\date{}

\begin{document}

\begin{abstract}
We study unital completely positive (UCP) maps on the finite Toeplitz operator system $\mathcal T_{\vd}$ of $\vd\times \vd$ Toeplitz matrices. We give an explicit characterization of pure UCP maps on $\mathcal T_{\vd}$ taking values in $M_n$ in terms of positive $n\times n$ matrix-valued trigonometric polynomials of degree at most $\vd-1$. This characterization yields a checkable criterion for deciding whether a given UCP map is pure. As an application, we prove that every pure UCP map on $\mathcal T_{\vd}$ has a unique completely positive extension to the generated $C^*$-algebra.
\end{abstract}


\maketitle

\tableofcontents

\section{Introduction}
Let $\mathcal S$ be an operator system. A state on $\mathcal S$ is a positive linear functional $\varphi:\mathcal S\to \mathbb C$ such that $\varphi(1)=1$. States are fundamental because the positivity and order-unit norm of $\mathcal S$ can be recovered by evaluating its elements against all states. However, an operator system carries more than its scalar order structure. It comes with a compatible family of positive cones $M_n(\mathcal S)^+$, $n\geq 1$. To detect this full matrix ordering, one is naturally led to matrix-valued states, namely unital completely positive (UCP) maps
\[
\varphi:\mathcal S\to M_n,\qquad n\geq 1.
\]

Among matrix states, the pure ones play a distinguished role. Pure UCP maps and pure matrix states have been studied from several related viewpoints. The foundational background goes back to Arveson’s seminal work on operator systems and boundary representations \cite{Arveson69, Arveson72}. From the perspective of matrix convexity, pure UCP maps coincide with matrix extreme points; see \cite{Farenick00,Farenick04}. Pure matrix states determine the matrix norm of every element in $M_{n}(\mathcal S)$ \cite{Kleski14}. Pure UCP maps also play an important role in the study of $C^{*}$-extreme points and $C^*$-convexity, where they appear as the basic irreducible building blocks in decompositions of $C^*$-extreme maps; see, for example, \cite{BK22, FM97, FZ98, LP81}.

\begin{definition}
A nonzero completely positive (CP) map $\varphi$ is said to be pure if every CP map dominated by $\varphi$ is a scalar multiple of $\varphi$; that is, whenever $\psi\leq_{\mathrm{cp}}\varphi$, one has
\[
\psi = t\varphi
\]
for some $t\in[0,1]$. Here $\psi\leq_{\mathrm{cp}}\varphi$ means that $\varphi-\psi$ is CP. Equivalently, a CP map is pure if it spans an extreme ray in the cone of CP maps.
\end{definition}

Pure CP maps have an intrinsic connection with Stinespring dilation theory. A CP map defined on a $C^{*}$-algebra is pure if and only if its minimal Stinespring dilation is irreducible. For operator systems,
however, this representation-theoretic criterion does not by itself give a
complete description. The restriction of a pure CP map to an operator subsystem
need not remain pure. Thus, unlike the $C^*$-algebraic case, the characterization
of pure CP maps is genuinely sensitive to the particular operator system under
consideration; see, for instance, \cite{CS21, Farenick04}.

\smallskip

There is another point of view on pure UCP maps. They play the role of noncommutative pure states for an operator system. At matrix level $n$, a matrix state on an operator system $\mathcal S$ is a UCP map $\varphi:\mathcal S\to M_n$, and the set $\bigsqcup_{n\geq 1}\operatorname{UCP}(\mathcal S,M_n)$ forms a matrix convex set \cite{EW97, WW99}.
Farenick \cite{Farenick00} showed that pure UCP maps are precisely the matrix extreme points of
this matricial state space. In other words, they cannot be
written as nontrivial matrix convex combinations of other matrix states. Hence, pure UCP maps are not merely extreme points in an ordinary convex set; they form a genuinely noncommutative boundary. It is therefore natural to seek a concrete characterization of them.

\subsection{Characterization of pure UCP maps} For arbitrary operator systems, even finite-dimensional ones, a concrete characterization of pure UCP maps is generally difficult. We study the finite Toeplitz operator systems, where positive matrix-valued trigonometric polynomials describe UCP maps; see \eqref{eq:P_phi} and Proposition \ref{prop:cp-density}. The finite Toeplitz setting therefore provides a tractable but nontrivial class in which pure matrix states can be described explicitly through polynomial densities and Fej\'er--Riesz factorization.

\smallskip

Finite Toeplitz operator systems arise classically from compressions of multiplication operators on $L^2(\mathbb T)$, where $\mathbb T$ denotes the unit circle. They encode finite sections of Toeplitz symbols, which makes their connection with trigonometric polynomials natural. Recently, Connes and van Suijlekom \cite{CS21} studied these systems as spectral truncations of the unit circle in the framework of noncommutative geometry \cite{Connes}. Farenick \cite{Farenick21} subsequently established a complete order isomorphism between finite Toeplitz systems and the dual of the space of scalar trigonometric polynomials.

\smallskip
	
Equation \eqref{eq:P_phi} associates to a CP map $\varphi$ a matrix-valued trigonometric polynomial $P_\varphi$, while Proposition \ref{prop:cp-density} shows that these polynomials serve as densities for the measures defining the corresponding CP maps. We call $P_\varphi$ the polynomial density of $\varphi$. Thus, the characterization of pure UCP maps becomes the problem of identifying the indecomposable polynomial densities among positive matrix-valued trigonometric polynomials. We carry this out for finite Toeplitz systems in Theorem \ref{thmA}.

\begin{customthm}\label{thmA}
Let $\varphi: \mathcal T_{\vd} \to M_{n}$ be a UCP map with polynomial density
$P_{\varphi}$. Then $\varphi$ is pure if and only if the following conditions hold:
\begin{enumerate}[leftmargin=*, label=(\arabic*), font=\normalfont]\itemsep=3pt
\item There exists a row polynomial
\[ 
Q(z) \,=\,\begin{pmatrix} q_1(z)&q_2(z)&\cdots&q_n(z)
	\end{pmatrix},
        \qquad q_j\in \mathbb C[z],\qquad \max_{j}\deg q_j = \vd-1,
\]
such that $P_\varphi(z)=Q(z)^*Q(z)$ for all $z\in\mathbb T$.
\item If $g_\varphi=\gcd(q_1,\ldots,q_n)$ denotes the greatest
common divisor of the polynomials appearing in condition $(1)$, then all zeros
of $g_\varphi$ lie on $\mathbb T$.
\end{enumerate}
\end{customthm}

\begin{remark}
We record several remarks concerning Theorem \ref{thmA}.

\begin{enumerate}[leftmargin=*, label=(\arabic*), font=\normalfont]\itemsep=5pt
\item[(a)] In Subsection \ref{ssec:checkable condition}, we explain how the characterization in Theorem \ref{thmA} leads to a practical criterion for deciding whether a given UCP map is pure. The discussion in that subsection also shows that the polynomial density of a pure UCP map admits a unique Fej\'er--Riesz factorization, up to left multiplication by a constant unit column vector.

\item[(b)]  For UCP maps of the form
\[
T\mapsto V^*TV, \qquad {T \in \mathcal T_{\vd}},
\]
where $V$ is an isometry, purity can be decided directly from $V$; see Subsection \ref{ssec:isometry}.

\item[(c)] Pure states ($n=1$ case) on $\mathcal T_{\vd}$ were characterized in \cite{CS21,Hekkelman22}. Our treatment is self-contained and approaches the result through polynomial densities and Fej\'er--Riesz factorization.
\end{enumerate}
\end{remark}
 
As an application of Theorem \ref{thmA}, we obtain a unique extension phenomenon for pure UCP maps, which we discuss below.

\subsection{Unique CP extension} Let $\mathcal{S}$ be an operator system and $\mathcal{H}$ be a Hilbert space. A CP map 
$ \varphi: \mathcal{S} \to \mathcal{B}(\mathcal{H}) $
is said to have a unique CP extension if there exists a unique CP map 
$ \tilde{\varphi}: \mathcal{C^*}(\mathcal{S}) \to \mathcal{B}(\mathcal{H}) $
such that $\tilde{\varphi}|_{\mathcal{S}} = \varphi$. The characterization of pure UCP maps leads to a unique extension phenomenon for pure UCP maps on finite Toeplitz systems. We first prove a general criterion for a CP map on $\mathcal T_{\vd}$ to have a unique CP extension to $C^{*}(\mathcal T_{\vd})=M_{\vd}$, expressed in terms of the Fej\'er--Riesz factorizations of its polynomial density. Combining this criterion with Theorem \ref{thmA} yields the following rigidity result.  This phenomenon is special to the finite Toeplitz setting, purity alone does not force unique CP extension for finite-dimensional operator systems in general, as shown by Example \ref{ex: counter example}.

\begin{customthm} \label{thmB}
If \( \varphi:\mathcal T_{\vd}\to M_n \) is a pure UCP map, then
\(\varphi\) has a unique CP extension to \(M_{\vd}\).
\end{customthm}

\begin{remark}
We next note some consequences and refinements of Theorem \ref{thmB}. 

\begin{enumerate}[leftmargin=*, label=(\arabic*), font=\normalfont]\itemsep=5pt
\item[(a)] More generally, if the polynomial density of a CP map admits a unique
Fej\'er--Riesz factorization up to left multiplication by a constant isometry,
then the CP map admits a unique CP extension; see Proposition \ref{prop:UCP}. 

\item[(b)] As a consequence of Theorem \ref{thmB}, we obtain the following uniqueness result. Let $\varphi$ be a pure UCP map and suppose that
\[
\varphi(T) \,=\, V^*TV \,=\, W^*TW,\qquad T\in\mathcal T_{\vd},
\]
for two isometries $V,W:\mathbb C^n\to\mathbb C^{\vd}$. Then
$ V=\lambda W $
for some $\lambda\in\mathbb T$; see Corollary \ref{cor:uniqueness of isometry}.


\end{enumerate}
\end{remark}

Unique extension phenomena for pure states on subspaces of $C^*$-algebras have been studied in \cite{Clouatre25}. A comparison between Theorem \ref{thmB} and Arveson’s unique extension property \cite{Arveson69, Arveson72} is given in Remark \ref{rem:comparison}.

\subsection{Organization of the paper}
The paper is organized as follows. In Section \ref{sec:preliminaries}, we recall the necessary background on finite Toeplitz operator systems, completely positive maps, and matrix-valued trigonometric polynomials. In Section \ref{sec:polynomial density}, we establish the connection between UCP maps on $\mathcal T_{\vd}$ and matrix-valued trigonometric polynomials. In Section \ref{sec:pure-ucp}, we prove the characterization of pure UCP maps on $\mathcal T_{\vd}$ stated in Theorem \ref{thmA}. We also explain how this characterization leads to a practical criterion for deciding whether a given UCP map is pure, and we discuss the case of UCP maps of the form $T\mapsto V^*TV$, where
$V:\mathbb C^n\to\mathbb C^{\vd}$ is an isometry. In Section \ref{sec:unique-extension}, we prove the unique CP extension theorem for pure UCP maps and give an example showing that this phenomenon is special to the Toeplitz setting.

\section{Preliminaries} \label{sec:preliminaries}

\subsection{Finite Toeplitz operator systems as spectral truncations of the circle}
For $\vd\geq 2$, let $\mathcal T_{\vd}\subseteq M_{\vd}$ denote the operator system consisting of all $\vd\times \vd$ Toeplitz matrices:
\[
\mathcal T_\vd
\, = \, 
\left\{ [a_{i-j}]_{i,j=0}^{\vd-1}:a_{-(\vd-1)},\ldots,a_{\vd-1}\in \mathbb C \right\}
\,\subseteq\, M_\vd .
\]
The $ C^*$-algebra generated by $ \mathcal T_\vd $ is $M_{\vd}.$ Note that 
$ \mathcal T_\vd \,=\,
{\rm span} \{I,J, \ldots,J^{\vd-1},J^*,\ldots,(J^*)^{\vd-1}\},$ where 
\[
J \, =\, 
\begin{bmatrix}
0&1&0&\cdots&0\\
0&0&1&\cdots&0\\
\vdots& &\ddots&\ddots&\vdots\\
0&0&\cdots&0&1\\
0&0&\cdots&0&0
\end{bmatrix}
\,\in\, M_\vd.
\]

\smallskip

Let $C(\mathbb T)$ denote the space of all continuous functions on the unit
circle. We regard $C(\mathbb T)$ as a subalgebra of
$\mathcal B(L^2(\mathbb T))$ by identifying $f\in C(\mathbb T)$ with the
multiplication operator $M_f$ on $L^2(\mathbb T)$. We shall simply write $f$
for $M_f$ whenever no confusion can arise.

\smallskip

For $k\in\mathbb Z$, let $e_k(z)=z^k,$ $z \in \mathbb T.$ Then $\{e_k:k\in\mathbb Z\}$ is an
orthonormal basis for $L^2(\mathbb T)$. Let $P_{\vd}$ denote the orthogonal
projection onto
$  H_{\vd}  := \operatorname{span}\{e_0,e_1,\ldots,e_{\vd-1}\}.$
For $f\in C (\mathbb T)$, the compression $P_{\vd}fP_{\vd}$ acts on
$H_{\vd}$. With respect to the basis $e_0,e_1,\ldots,e_{\vd-1}$, it is
represented by the $\vd\times \vd$ Toeplitz matrix
\[
\begin{bmatrix}
\widehat f(0) & \widehat f(-1)  & \cdots & \widehat f(-\vd+1) \\
\widehat f(1) & \widehat f(0) & \cdots & \widehat f(-\vd+2) \\
\vdots & \vdots & \ddots  & \vdots \\
\widehat f(\vd-1) & \widehat f(\vd-2) & \cdots & \widehat f(0)
\end{bmatrix}.
\]
Equivalently, since $J=P_{\vd}M_{\overline z}P_{\vd}$, we can also write
\[
        P_{\vd}fP_{\vd}
        =
        \widehat f(0)I
        +
        \sum_{k=1}^{\vd-1}\widehat f(-k)J^k
        +
        \sum_{k=1}^{\vd-1}\widehat f(k)(J^*)^k .
\]
Thus the finite Toeplitz operator system $\mathcal T_{\vd}\subseteq M_{\vd}$
may be realized as
$  \mathcal T_{\vd}
        =
        \{P_{\vd}fP_{\vd}:f\in C(\mathbb T)\}.$
Indeed, every $\vd\times\vd$ Toeplitz matrix arises in this way by choosing a
trigonometric polynomial $f$ of degree at most $\vd-1$ whose Fourier coefficients agree with the entries
on the relevant diagonals. Such spectral truncations have been studied in
detail in \cite{CS21, Hekkelman22, W}.

\subsection{Operator systems and completely positive maps}
Let $\mathcal S$ be an operator system. We write
$\operatorname{UCP}(\mathcal S,M_n)$ for the set of all the UCP
maps from $\mathcal S$ to $M_n$. 
For CP maps $\psi,\varphi:\mathcal S\to M_n$, we write $\psi\leq_{\mathrm{cp}}\varphi$ if $\varphi-\psi$ is CP.

\smallskip

The following extension result of Arveson \cite[p.~180]{Arveson69}
(see also \cite[Theorem B]{Farenick00}) is one of the main technical inputs in
our characterization of pure UCP maps. It allows us to pass from pure UCP maps
on an operator system to pure UCP maps on the generated $C^*$-algebra, thereby
providing a necessary condition for purity in the operator system setting.

\begin{theorem*}[Pure extension theorem]
Let $\mathcal S \subseteq \mathcal A = C^*(\mathcal S)$ be an operator system. If $\varphi:\mathcal S \to M_n$ is a pure UCP map, then $\varphi$ admits a pure UCP extension $\tilde{\varphi}:\mathcal A\to M_n$.
\end{theorem*}

In our setting, the operator system under consideration is the finite Toeplitz
system $\mathcal T_{\vd}$, and $C^*(\mathcal T_{\vd})=M_{\vd}$. Pure UCP maps
$M_{\vd}\to M_n$ have a particularly concrete form, which follows as an
elementary consequence of Stinespring dilation theory.

\begin{lemma}\label{lem:pureucp}
Let $\varphi:M_{\vd} \to M_n$ be a UCP map. Then $\varphi$ is pure
if and only if there exists an isometry
$ V:\mathbb C^n\to \mathbb C^{\vd} $
such that
\[
\varphi(T) \ =\ V^*TV,\qquad T\in M_{\vd}.
\]
In particular, a pure UCP map $M_{\vd}\to M_n$ can exist only if
$n\leq \vd$.
\end{lemma}

\begin{proof}
The $C^{*}$-algebra $ M_\vd$ has, up to unitary equivalence, only one irreducible representation, namely the identity representation on $ \mathbb C^\vd$. A pure CP map on a $C^{*}$-algebra has irreducible minimal Stinespring representation. Hence, if $ \varphi $ is pure, its minimal Stinespring representation is unitarily equivalent to the identity representation of $ M_\vd$, and therefore
\[
\varphi(T) \,=\, V^* T V,\qquad (T \in M_\vd),
\]
for some operator	$ V:\mathbb C^n\to\mathbb C^\vd. $
Since $\varphi$ is unital,
$ I_n \,=\, \varphi(I_\vd) \,=\, V^*V. $
Thus $ V $ is an isometry.
		
\smallskip
		
Conversely, if	$ \varphi(T) = V^* T V $ with $ V^*V = I_n $, then the Stinespring representation is the identity representation of $ M_\vd$, which is irreducible and minimal. Hence $ \varphi$ is a pure map.
\end{proof}

The next result is another fundamental tool used throughout the paper. In our
setting, CP maps on the finite Toeplitz system are often studied through their
extensions to the full matrix algebra $M_{\vd}$. Choi's theorem then converts
complete positivity of such extensions into an explicit positivity condition
of a finite block matrix.

\begin{theorem*}[Choi's theorem \cite{Choi75}]
Let $\varphi:M_\vd \to M_n $ be a linear map, and let
$\{E_{ij}\}_{i,j=0}^{\vd-1}$ denote the standard matrix units in $M_\vd$.
Then $\varphi$ is CP if and only if its {\em Choi matrix}
\[
C_\varphi \,=\, \bigl[\varphi(E_{ij})\bigr]_{i,j=0}^{\vd-1}
        \in M_\vd(M_n)\cong M_{\vd\,n}
\]
is positive semidefinite.
\end{theorem*}

For background on operator systems, CP maps,
$C^*$-algebras, and their representations, we refer the reader to
\cite{David25,Paulsen}.

\subsection{Positive matrix-valued trigonometric polynomials}
Let $\mathbb T=\{z\in\mathbb C: |z|=1\}$. Let
$P:\mathbb T\to M_n$ be a matrix-valued trigonometric polynomial of the form
\begin{equation} \label{eq:trig poly}
P(z)\,=\,
\sum_{k=-\vd}^{\vd} A_k z^k, \qquad A_k\in M_n.
\end{equation}
We say that $P$ is {\em positive} if $P(z)\succeq 0$ for all
$z\in\mathbb T$, where, for an operator $T$ on a Hilbert space $\mathcal H$,
the notation $T\succeq 0$ means that $T$ is positive semidefinite. We say that $P$ has degree at most $\vd$ if it is of the form \eqref{eq:trig poly}. If $P$ is not identically zero, its degree is the largest integer $|k|$ for
which $A_k\neq 0$.

\smallskip

A positive matrix-valued trigonometric polynomial $P$ is said to be
{\em normalized} if its constant coefficient is $I_n$. Equivalently,
\[
        \int_{\mathbb T} P(z)\,dm(z)=I_n,
\]
where $dm$ denotes the normalized arc length measure on $\mathbb T$.

\smallskip

The following result is one of the central tools of the paper. The scalar-valued Fej\'er--Riesz factorization asserts that every non-negative scalar-valued trigonometric polynomial of degree at most $\vd$ can be written as the modulus square of an analytic polynomial of degree at most $\vd$. Its matrix-valued and operator-valued extensions, due in particular to Rosenblum \cite{Rosenblum68}, play a fundamental role in Toeplitz operator theory and factorization theory. For a modern account and further perspectives on the operator Fej\'er--Riesz factorization, see \cite{DR10}. We shall use the matrix-valued version repeatedly to
pass from positive matrix-valued trigonometric polynomials to analytic
polynomial factorizations.

\begin{theorem*}[Matrix-valued Fej\'er--Riesz theorem]
Let $P$ be an $n\times n$ matrix-valued positive trigonometric polynomial of degree $\vd$. Then 
\[
P(z) \,=\,
Q(z)^*Q(z),
\]
where $Q$ is an analytic $m \times n$ matrix-valued polynomial of degree $\vd$. 
\end{theorem*}

\begin{definition}
Let $P$ be an $n\times n$ matrix-valued positive trigonometric polynomial of degree $\vd$. A Fej\'er--Riesz factor of $P$ is an $m \times n$ matrix-valued analytic polynomial $Q$ of degree $\vd$ such that
\[
P(z) \,=\, Q(z)^*Q(z),\qquad z\in\mathbb T.
\]
Here, the number of rows $m$ is allowed to depend on $Q$. 
\end{definition}

\section{UCP Maps via Polynomial Densities}\label{sec:polynomial density}
Let	$ \varphi:\mathcal T_\vd\to M_n $
	be a linear map. Define an  $n \times n$ matrix-valued trigonometric
	polynomial
	\begin{equation} \label{eq:P_phi}
	P_\varphi(z)
	\,=\,
	\varphi(I)
	+
	\sum_{k=1}^{\vd-1}\varphi(J^k)z^k
	+
	\sum_{k=1}^{\vd-1}\varphi((J^*)^k)z^{-k},
	\qquad (z\in\mathbb T).
	\end{equation}
	Let $C(\mathbb{T})_{(\vd-1)}^{(n)}$ be the vector space of $n \times n$ matrix-valued trigonometric polynomials of degree at most $\vd-1$ and $\mathcal{L}(\mathcal{T}_{\vd}, M_{n})$ be the vector space of all linear maps $ \varphi:\mathcal T_\vd\to M_n $. 
    
    \smallskip

The following result can alternatively be deduced from the complete order
isomorphism between $\mathcal T_{\vd}$ and the dual of
$C(\mathbb T)_{\vd-1}^{(1)}$ established in
\cite[Theorem 1.1]{Farenick21}; see also
\cite[Proposition 4.6]{CS21}. However, our proof is independent of this
duality result. We give a proof using
Arveson's extension theorem, the matrix-valued Fej\'er--Riesz factorization, and Choi's theorem.

\begin{lemma} \label{lem:bijection}
The map
\[
\Phi_{n}: \mathcal{L}(\mathcal{T}_{\vd}, M_{n}) \to C(\mathbb{T})_{(\vd-1)}^{(n)}; \qquad \varphi \mapsto P_{\varphi}
\]
is a linear bijective map such that $ \varphi \in \mathcal{L}(\mathcal{T}_{\vd}, M_{n}) $ is CP if and only if $P_{\varphi} \in C(\mathbb{T})_{(\vd-1)}^{(n)} $ is positive on the unit circle. Moreover, $\varphi$ is UCP if and only if $P_{\varphi}$ is normalized. 
\end{lemma}
	
\begin{proof}
The injectivity of $\Phi_{n}$ is immediate from the definition. The surjectivity follows from the fact that 
\[ 
{\rm dim}\, \mathcal{L}(\mathcal{T}_{\vd}, M_{n}) 
\,=\,
(2 \vd-1) n^{2} 
\,=\,
{\rm dim}\, C(\mathbb{T})_{(\vd-1)}^{(n)}.
\]
		
For the remaining part of the assertion, let $ \varphi \in \mathcal{L}(\mathcal{T}_{\vd}, M_{n})$ be a CP map. By the Arveson extension theorem, $ \varphi$ extends to a CP map $ \tilde{\varphi}:M_\vd\to M_n. $ Since $ \tilde{\varphi}$ is CP, its Choi matrix
$ C_{\tilde{\varphi}} = [\tilde{\varphi}(E_{ij})]_{i,j=0}^{\vd-1} \in M_\vd(M_n) $
is positive. Thus there exist  matrices $ Q_0,Q_1,\ldots,Q_{\vd-1} $ of suitable size such that
\[
\tilde{\varphi} (E_{ij}) 
\, =\, 
Q_i^*Q_j
\qquad
0\leq i,j\leq \vd-1.
\]
This is just a Gram factorization of the positive block matrix 
$ [\tilde{\varphi}(E_{ij})]_{i,j=0}^{\vd-1}.$
Define
\[
Q(z) \,=\, Q_0 +Q_1z+\cdots+Q_{\vd-1}z^{\vd-1}.
\]
Then
$ Q(z)^*Q(z) = \sum_{i,j=0}^{\vd-1}Q_i^*Q_jz^{j-i}. $
We now compute the Fourier coefficients of $ Q(z)^*Q(z).$ First, the constant coefficient is
		\[
		\sum_{i=0}^{\vd-1}Q_i^*Q_i
		\,=\,
		\sum_{i=0}^{\vd-1}\tilde{\varphi}(E_{ii})
		\,=\,
		\tilde{\varphi}(I)
		\,=\,
		\varphi(I).
		\]
		Next, for $ 1\leq k\leq \vd-1$, the coefficient of $z^k$ in $ Q(z)^*Q(z)$ is 
		$ \sum_{i=0}^{\vd-1-k} Q_i^*Q_{i+k}. $
		Using
		$ J^k = \sum_{i=0}^{\vd-1-k}E_{i,i+k}, $
		we get
		\[
		\sum_{i=0}^{\vd-1-k}Q_i^*Q_{i+k}
		\,=\,
		\sum_{i=0}^{\vd-1-k}\tilde{\varphi}(E_{i,i+k})
		\,=\,
		\tilde{\varphi}(J^k)
		\,=\,
		\varphi(J^k).
		\]
		Similarly, the coefficient of $z^{-k}$ is
		\[
		\sum_{i=0}^{\vd-1-k}Q_{i+k}^*Q_i
		\,=\,
		(\tilde{\varphi}(J^k))^{*}
		\,=\,
		\varphi((J^*)^k).
		\]
		Therefore
		\[
		Q(z)^*Q(z)
		\,=\,
		\varphi(I)
		+
		\sum_{k=1}^{\vd-1}\varphi(J^k)z^k
		+
		\sum_{k=1}^{\vd-1}\varphi((J^*)^k)z^{-k}
		\,=\,
		P_\varphi(z).
		\]
		Thus
		$ P_\varphi(z) \, =\, Q(z)^*Q(z)\geq 0 $
		for all $ z\in\mathbb T$.
		
		\smallskip
		
		Conversely, assume that
		$ P_\varphi(z)\geq 0 $ for all 
		$ z\in\mathbb T$ for some $\varphi \in \mathcal L (\mathcal T_{\vd} , M_{n}).$
		By the matrix-valued Fej\'er--Riesz factorization, there exist matrices
		$ R_0,R_1,\ldots,R_{\vd-1} $ of a common suitable size
		such that 
		$P_\varphi(z)\,=\,R(z)^*R(z), $
		where
		$ R(z) = R_0+R_1z+\cdots+R_{\vd-1}z^{\vd-1}. $
		Thus
		$ P_\varphi(z)
		=
		\sum_{i,j=0}^{\vd-1}R_i^*R_jz^{j-i}. $ Define a linear map $\psi_{R}:M_\vd\to M_n$
		on matrix units by
		$ \psi_R(E_{ij}) \,=\, R_i^*R_j $ for 
		$ 0\leq i,j\leq \vd-1.$
		Its Choi matrix is
		\[
		C_{\psi_{R}} \,=\, [\psi_{R} (E_{ij})]_{i,j=0}^{\vd-1}
		\,=\,
		[R_i^*R_j]_{i,j=0}^{\vd-1}.
		\]
		This block matrix is positive. By Choi's theorem, $\psi_{R}$ is CP. It remains to show that
		$ \psi_{R}|_{\mathcal T_\vd} = \varphi. $
		The constant coefficient of $ P_\varphi(z)$ is $ \varphi(I).$ On the other
		hand, the constant coefficient of $ R(z)^*R(z)$ is
		$ \sum_{i=0}^{\vd-1}R_i^*R_i. $
		Therefore
		\[
		\varphi(I) \,=\, \sum_{i=0}^{\vd-1} R_i^*R_i
		\,=\,
		\sum_{i=0}^{\vd-1} \psi_R (E_{ii})
		\,=\,
		\psi_R(I).
		\]
		For $ 1\leq k\leq \vd-1 $, the coefficient of $ z^k $ in $ P_\varphi(z)$ is	$ \varphi(J^k).$
		The coefficient of $ z^k$ in $ R(z)^*R(z)$ is
		$ \sum_{i=0}^{\vd-1-k} R_i^*R_{i+k}. $
		Therefore	$ \varphi(J^k)	=	\sum_{i=0}^{\vd-1-k} R_i^*R_{i+k}. $
		But
		$ J^k = \sum_{i=0}^{\vd-1-k} E_{i,i+k}. $
		Hence
		\[
		\varphi(J^k)
        \,=\,
		\sum_{i=0}^{\vd-1-k}R_i^*R_{i+k}
        \,=\,
		\sum_{i=0}^{\vd-1-k}\psi_R(E_{i,i+k})
        \,=\,
        \psi_R(J^k).
		\]
		Similarly, comparing the coefficients of $ z^{-k}$, we get
		$ \psi_R((J^*)^k) \,=\, \varphi((J^*)^k). $
		This completes the proof. 
	\end{proof}

The following proposition gives an integral representation for CP maps on $\mathcal T_{\vd}$. It shows that every CP map $\varphi:\mathcal T_{\vd}\to M_n$ is represented by the positive matrix-valued trigonometric polynomial $P_{\varphi}$ defined in \eqref{eq:P_phi}. We call this polynomial the polynomial density of $\varphi$.

\begin{prop}\label{prop:cp-density}
Let $\varphi:\mathcal T_{\vd}\to M_n$ be a CP map. Then
\[
\varphi(P_{\vd} f P_{\vd})
\,=\,
\int_{\mathbb T} f(z)  P_{\varphi} (z)\, dm(z)
\]
for every $f\in C(\mathbb T),$ where $P_{\varphi}$ is defined in \eqref{eq:P_phi}.
\end{prop}

\begin{proof}
By Arveson's extension theorem, $\varphi$ extends to a CP map
$\tilde \varphi:M_{\vd}\to M_n$. Hence, by the finite-dimensional Stinespring--Kraus
representation (Choi's theorem), there exist operators
$V_1,\ldots,V_r:\mathbb C^n\to\mathbb C^{\vd}$ such that
\[
\tilde \varphi (T) 
\,=\, \sum_{\ell=1}^r V_\ell^* T V_\ell,
\qquad T\in M_{\vd}.
\]

The $(i,j)$-entry of
$P_{\vd} f P_{\vd}$, for $0\leq i,j\leq \vd-1$, is
$ \widehat f(i-j) = \int_{\mathbb T} f(z)z^{j-i}\,dm(z).$
Now define
\[
\eta(z) \,=\,
\begin{pmatrix}
1 & \overline z & \overline z^2 \cdots  \overline z^{\vd-1}
\end{pmatrix}^{t}
\in \mathbb C^{\vd}.
\]
Then the $(i,j)$-entry of $\eta(z)\eta(z)^*$ is $z^{j-i}$.
Therefore the $(i,j)$-entry of 
$ \int_{\mathbb T} f(z)\eta(z)\eta(z)^*\,dm(z) $
is also $\int_{\mathbb T}f(z)z^{j-i}\,dm(z)$. Hence
$ P_{\vd} f P_{\vd}
        \,=\,
        \int_{\mathbb T} f(z)\eta(z)\eta(z)^*\,dm(z). $
Therefore
\[
\begin{aligned}
        \varphi(P_{\vd} f P_{\vd})
        & \,=\,
        \tilde \varphi(P_{\vd} f P_{\vd})                                      \\
        & \,=\,
        \sum_{\ell=1}^r V_\ell^*
        \left(\int_{\mathbb T} f(z)\eta(z)\eta(z)^*\,dm(z)\right)
        V_\ell                                                        \\
        &=
        \int_{\mathbb T} f(z)
        \left(\sum_{\ell=1}^r V_\ell^*\eta(z)\eta(z)^*V_\ell\right)
        dm(z).
\end{aligned}
\]
Thus the required density is
\[
\widetilde P_{\varphi}(z)
\,:=\,
\sum_{\ell=1}^r V_\ell^*\eta(z)\eta(z)^*V_\ell .
\]
This is a positive $M_n$-valued trigonometric polynomial of degree at most
$\vd-1$. We need to show that $\widetilde P_{\varphi} = P_{\varphi}.$  Write $\widetilde P_\varphi(z) = \sum_{k=-(\vd-1)}^{\vd-1} A_k z^k .$
Since $J = P_{\vd} \overline{z} P_{\vd}$, we have $ J^k \,=\, P_{\vd} \overline z^k P_{\vd} $ for $|k| \leq \vd -1.$
Hence
\[
\varphi(J^k)
\,=\,
\int_{\mathbb T} z^{-k} \widetilde P_\varphi(z)\,dm(z) \qquad (|k| \leq \vd-1).
\]
Therefore
\[
\varphi(J^k)
\,=\,
\sum_{\ell =-(\vd-1)}^{\vd-1}A_\ell
\int_{\mathbb T} z^{\ell-k}\,dm(z)
\,=\,
A_k.
\]
Thus $\varphi(J^k)$ is the Fourier coefficient of $\widetilde P_\varphi$ corresponding to $z^k$. Therefore $\widetilde P_\varphi = P_{\varphi}.$ 

\end{proof}

\section{Characterization of pure UCP maps}\label{sec:pure-ucp}
Let $\mathcal{P}_{(\vd-1)}^{+}(\mathbb T)$ denote the cone of scalar-valued positive trigonometric polynomials of degree at most $\vd-1$. We begin with a lemma which characterizes those elements of this cone that span extreme rays. Under the polynomial density correspondence, Proposition \ref{prop:cp-density}, this is equivalent to characterizing pure states on $\mathcal T_{\vd}$. This result should be compared with \cite[Proposition 4.8]{CS21} and \cite[Proposition 2.12]{Hekkelman22}. The proofs in these works implicitly contain the degree condition (see Lemma \ref{lem:extreme ray}), although it is not stated explicitly in the assertions. We give a different proof below, making the degree condition explicit. Our proof uses only the scalar Fej\'er--Riesz factorization.
	
\begin{lemma} \label{lem:extreme ray}
Let $ g\in\mathbb C[z] $ be a nonzero polynomial of degree at most $ \vd-1 $.
Then the positive trigonometric polynomial
$ |g|^2 $
spans an extreme ray of $ \mathcal{P}_{(\vd-1)}^{+}(\mathbb T) $ if and only if $ \deg g = \vd-1 $
and every zero of $g$ lies on the unit circle.
\end{lemma}
	
\begin{proof}
First suppose $ \deg g< \vd-1.$	Then the polynomial
 $\tilde g(z) \ =\ zg(z) $
 has degree at most $\vd-1$, and
$ |\tilde g(z)| \,=\, |g(z)| $ on $ \mathbb T. $ Hence
\[
|g|^2
\,=\,
\left|\frac{g+\tilde g}{2}\right|^2+\left|\frac{g-\tilde g}{2}\right|^2.
\]
The two summands are positive trigonometric polynomials of degree at most
$\vd-1$. They are not scalar multiples of $|g|^2$, because
$g+\tilde g$ and $g-\tilde g$ are not scalar multiples of
$g$. Thus $|g|^2$ does not span an extreme ray.
		
\smallskip
		
Next suppose $g$ has a zero $\alpha\notin\mathbb T$. Write
$ g(z)=(z-\alpha)h(z). $
Define
$ \tilde g(z)  = (1-\overline{\alpha}z)h(z). $
For $z \in \mathbb T$, one has $	|z-\alpha|=|1-\overline{\alpha}z|. $
Therefore 
$ |\tilde g(z)|  = |g(z)|,$ for all $z \in \mathbb T.$
Again,
	\[
	|g|^2
	\,=\,
	\left|\frac{g+\tilde g}{2}\right|^2+\left|\frac{g-\tilde g}{2}\right|^2.
	\]
is a nontrivial decomposition inside $\mathcal{P}_{(\vd-1)}^{+}(\mathbb T)$. Hence
$|g|^2$ is not extreme.
		
\smallskip
		
Conversely, assume $\deg g= \vd-1 $ and every zero of $ g $ lies on $ \mathbb T $. Let
$ 0 \leq  p  \leq   |g|^2 $ with $ p\in  \mathcal{P}_{(\vd-1)}^{+}(\mathbb T) $. By the Fej\'er--Riesz factorization, there is a polynomial $ h \in \mathbb C[z]$ of degree at most $\vd-1$ such that	$ p (z)=|h(z)|^2.$
The inequality
$ |h(z)|^2  \leq  |g(z)|^2 $
implies 	$ |h(z)|\leq |g(z)| $
for all $ z \in \mathbb T$. If $ \lambda\in\mathbb T $ is a zero of $ g$ of
multiplicity $ m$, then for $z\neq \lambda$ we get
\[
\frac{|h(z)|}{|z-\lambda|^m} \,\leq\, \frac{|g(z)|}{|z-\lambda|^m}
\]
Since the right-hand side has a finite limit as $z \in \mathbb{T}$ tends to $\lambda,$ the left-hand side remains bounded in a neighbourhood of $\lambda.$ Hence $ h$ must vanish at $ \lambda $ with multiplicity at least $ m $. Since all zeros of $ g $ lie on $ \mathbb T $, and
$ \deg h\leq \vd-1 = \deg g,$
it follows that 
$	h=cg $
for some scalar $c\in\mathbb C$. Hence
$	p=|c|^2|g|^2. $
Therefore the only positive trigonometric polynomials dominated by
$|g|^2$ are scalar multiples of $|g|^2$, and so $|g|^2$ spans an
extreme ray.
\end{proof}

In the proposition below, we characterize pure UCP maps $\tilde{\varphi}:M_{\vd}\to M_n$ by looking at the polynomial density of their restriction to $\mathcal T_{\vd}$. Thus the purity of a UCP map on the full matrix algebra is translated into a condition on the polynomial density associated with $\tilde{\varphi}|_{\mathcal T_{\vd}}$.

\begin{proposition}\label{prop:compression-density}
Let $\varphi:\mathcal T_{\vd}\to M_n$ be a UCP map and let $P_\varphi$ be its
associated polynomial density. Then the following are equivalent.

\begin{enumerate}[leftmargin=*, label=(\arabic*), font=\normalfont]\itemsep=3pt
\item There exists an isometry $V:\mathbb C^n\to \mathbb C^{\vd}$ such that
\[
\varphi(T) \,=\, V^*TV, \qquad T\in \mathcal T_{\vd}.
\]
\item There exists a row polynomial
\begin{equation} \label{eq:Qdef}
Q(z) \,=\,\begin{pmatrix} q_1(z)&q_2(z)&\cdots&q_n(z)
	\end{pmatrix},
        \qquad q_j\in \mathbb C[z],\quad \deg q_j\leq \vd-1,
\end{equation}
such that $ P_\varphi(z)=Q(z)^*Q(z)$ for all $z\in\mathbb T$.
\end{enumerate}
\end{proposition}
\begin{proof}
Let
$\eta(z)= \begin{pmatrix} 1 &  \overline z & \cdots & \overline z^{\vd-1} \end{pmatrix}^{t} \in \mathbb C^{\vd}.$
If $\varphi(T) = V^*TV$, then, using
\[
P_{\vd} f P_{\vd}
\,=\,
\int_{\mathbb T} f(z)\eta(z)\eta(z)^*\,dm(z),
\]
we obtain
\[
\varphi(P_{\vd} f P_{\vd})
 \,=\,
 V^*\left(
        \int_{\mathbb T} f(z)\eta(z)\eta(z)^*\,dm(z)
        \right)
        V 
         \,=\,
        \int_{\mathbb T} f(z)V^*\eta(z)\eta(z)^*V\,dm(z).
\]
Hence
$ P_\varphi(z) = V^*\eta(z)\eta(z)^*V. $ If we put
$ Q(z) = \eta(z)^*V, $ then $Q$ is a row polynomial of degree at most $\vd-1$, and
$  P_\varphi(z)=Q(z)^*Q(z).$

\smallskip

Conversely, suppose
\[
P_\varphi(z) \,=\, Q(z)^*Q(z),
\qquad
        Q(z) \,=\, Q_0+Q_1z+\cdots+Q_{\vd-1}z^{\vd-1}, \qquad Q_{j} \,\in\, M_{1,n}.
\]
Define
\[
V \,=\,
        \begin{pmatrix}
        Q_0\\
        Q_1\\
        \vdots\\
        Q_{\vd-1}
        \end{pmatrix}
        :
        \mathbb C^n\to \mathbb C^{\vd}.
\]
Then
\[
        V^*V  \,=\, \sum_{j=0}^{\vd-1}Q_j^*Q_j
        \,=\,
        \int_{\mathbb T}Q(z)^*Q(z)\,dm(z)
        \,=\,
        \int_{\mathbb T} P_{\varphi} (z) \,dm(z)
        \, =\,
        I_n,
\]
so $V$ is an isometry. Also $Q(z)=\eta(z)^*V$.
$  P_\varphi(z) = V^*\eta(z)\eta(z)^*V.$
Therefore, for every $f\in C(\mathbb T)$,
\[
\varphi(P_{\vd} f P_{\vd})
\,=\,
\int_{\mathbb T}f(z)P_\varphi(z)\,dm(z)
\,=\,
V^* \left( \int_{\mathbb T}f(z)\eta(z)\eta(z)^*\,dm(z) \right) V
\,=\, V^*P_{\vd} f P_{\vd}V.
\]
Thus $\varphi(T) = V^*TV$ for every $T\in\mathcal T_{\vd}$.
\end{proof}

\subsection{A Guiding Example}
Lemma \ref{lem:extreme ray} and Proposition \ref{prop:compression-density} suggest a tempting but false criterion for purity of UCP maps from $\mathcal T_{\vd}$ to $M_n$. One might expect that a UCP map $\varphi:\mathcal T_{\vd}\to M_n$ is pure whenever its polynomial density admits a factorization $P_\varphi=Q^*Q$, with $Q$ as in \eqref{eq:Qdef}, such that $\max_j\deg q_j=\vd-1$ and each scalar polynomial $q_j$ has all its zeros on $\mathbb T$. The following example shows that this criterion is not necessary in the matrix-valued setting.

\begin{example}
Let $Q(z)=(1,z)$ and
$ P_\varphi(z) = Q(z)^*Q(z)
=
  \begin{pmatrix}
        1 & z\\
        z^{-1} & 1
        \end{pmatrix}.$
Let $\varphi:\mathcal T_2\to M_2$ be the UCP map with polynomial density
$P_\varphi$. Then $\varphi$ is pure.
\end{example}
\begin{proof}
Let $\psi:\mathcal T_2\to M_2$ be CP with
$\psi\leq_{\mathrm{cp}}\varphi$. Let $P_\psi$ be the polynomial density of
$\psi$. Since $P_\varphi = Q^*Q$ with $Q(z)=(1,z)$, we have
\[
0 \,\preceq\, P_\psi(z) \,\preceq\, Q(z)^*Q(z),
\qquad z\in\mathbb T.
\]
Set
$  u(z)=
        \begin{pmatrix}
        -z\\
        1
        \end{pmatrix}.$
Then $Q(z)u(z)=0$. Hence
\[
        0 \,\leq\, \left\langle P_\psi(z)u(z),u(z) \right\rangle
        \,\leq\, \left\langle Q(z)^*Q(z)u(z),u(z) \right\rangle
        \,=\,
        \|Q(z)u(z)\|^2
        \,=\,
        0.
\]
Since $P_\psi(z)\succeq 0$, it follows that $P_\psi(z)u(z)=0$ for every
$z\in\mathbb T$. Write
\[
        P_\psi(z)=
        \begin{pmatrix}
        a(z) & b(z)\\
        b(z)^* & c(z)
        \end{pmatrix}.
\]
The equation $P_\psi(z)u(z) = 0$ gives
$ b(z)=za(z).$
Since $P_\psi$ has degree at most $1$, write
\[
        a(z) \,=\, \alpha+\beta z+\overline{\beta}z^{-1}.
\]
Then
\[
    b(z) \,=\, za(z) \,=\, \alpha z + \beta z^2+\overline{\beta}.
\]
But $b$ also has degree at most $1$, so $\beta=0$. Hence $a(z)=\alpha$ is constant.
Consequently,
\[
        b(z)=\alpha z,
        \qquad
        c(z)=\alpha,
\]
and therefore
\[
        P_\psi(z)
        \,=\,
        \alpha
        \begin{pmatrix}
        1 & z\\
        z^{-1} & 1
        \end{pmatrix}
        \,=\,
        \alpha\, Q(z)^*Q(z).
\]
Since $0\preceq P_\psi\preceq P_\phi$, we have $\alpha\in[0,1]$. Thus
$P_\psi=\alpha P_\varphi$, and hence $\psi=\alpha\varphi$. Therefore $\varphi$ is pure.
\end{proof}

The following proposition is the main ingredient in the proof of Theorem \ref{thmA}.

\begin{prop}\label{prop:domination}
Let $\varphi:\mathcal T_{\vd}\to M_n$ be a UCP map, and suppose that its
polynomial density has the form
$ P_\varphi(z)=Q(z)^*Q(z),$
where $Q$ is as in $\eqref{eq:Qdef}.$
Write
\[
Q(z)\,=\,g_\varphi(z)R(z),
\qquad
R(z) \,=\, \begin{pmatrix} r_1(z)&r_2(z)&\cdots&r_n(z)\end{pmatrix},
\]
where $g_\varphi=\gcd(q_1,\ldots,q_n)$ is chosen so that
$\gcd(r_1,\ldots,r_n)=1$. Let
$ \vd_r=\max_{1\leq j\leq n}\deg r_j. $
Then a CP map $\psi:\mathcal T_{\vd}\to M_n$ satisfies
$\psi\leq_{\mathrm{cp}}\varphi$ if and only if there exists a scalar polynomial
$h\in\mathbb C[z]$ with $\deg h\leq \vd-1-\vd_r$ such that
\[
        P_\psi(z) \,=\, |h(z)|^2 R(z)^*R(z),
        \qquad z\in\mathbb T,
\]
and
\[
        0 \,\leq\, |h(z)|^2 \,\leq\, |g_\varphi(z)|^2,
        \qquad z\in\mathbb T.
\]
\end{prop}
	
\begin{proof}
The reverse implication is straightforward. It remains to prove the forward implication. Since $Q(z) = g_\varphi(z)R(z)$,
		we have
		\[
		P_{\varphi} (z) \,=\, Q(z)^*Q(z)
		\,=\,
		|g_\varphi(z)|^2 R(z)^*R(z).
		\]
Since $\psi$ is CP, $P_{\psi}$ is positive. By the matrix-valued Fej\'er--Riesz factorization $P_{\psi}(z)  =  F(z)^{*}F(z),$ where $ F $ is an $m \times n $ matrix-valued polynomial of degree at most $\vd-1$. Since $ 0 \leq_{\mathrm{cp}} \psi \leq_{\mathrm{cp}} \varphi,$ we have $0 \preceq P_{\psi} \preceq P_{\varphi}$ and thus 
\[
0 \,\preceq\, F(z)^*F(z) \,\preceq \, Q(z)^*Q(z) \qquad (z \in \mathbb T).
\]
This implies that for every $ z\in\mathbb T$, the range of
$ F(z)^*F(z) $ is contained in the range of $ Q(z)^*Q(z)$.  But 
\[
{\rm Range} (Q(z)^{*}Q(z)) \,=\,  {\rm Range} (Q(z)^{*}) ,
\]
and
\[
{\rm Range} (F(z)^{*}F(z)) \,=\,  {\rm Range} (F(z)^{*}) .
\]
Since
$Q(z)^*Q(z)$ has rank at most one, every row of $ F(z)$ must be
pointwise proportional to $ Q(z)$. Write a row of $ F$ as
$ f(z)= \begin{pmatrix}
			f_1(z)&\cdots&f_n(z)
		\end{pmatrix}. $
The fact that $f(z)$ is pointwise proportional to $Q(z)=g_\varphi(z)R(z)$
is equivalent to
		\[
		f_i(z)r_j(z) \,=\, f_j(z)r_i(z)
		\qquad
		\text{for all }i,j.
		\]
Since $\gcd(r_1,\ldots,r_n)=1$, by B\'ezout's identity for univariate polynomials with coefficients in a field, there exist polynomials
$a_1,\ldots,a_n\in \mathbb C[z]$ such that
\[
        \sum_{j=1}^n a_j(z)r_j(z) \,=\, 1.
\]
Using the identities
\[
        f_i(z)r_j(z)=f_j(z)r_i(z),
        \qquad 1\leq i,j\leq n,
\]
we obtain
\[
        f_i(z)
        =
        f_i(z)\sum_{j=1}^n a_j(z)r_j(z)
        =
        \sum_{j=1}^n a_j(z)f_i(z)r_j(z)
        =
        \sum_{j=1}^n a_j(z)f_j(z)r_i(z).
\]
If we set
$ \alpha(z) =  \sum_{j=1}^n a_j(z)f_j(z), $
then
$    f_i(z) \,=\, \alpha(z) r_i(z)$ for  $ 1\leq i\leq n.$
Thus
$ f(z) \,=\, \alpha(z)R(z).$
Applying this argument to each row of $F$, we obtain
\[
F(z) \,=\, A(z)R(z),
\]
where $A$ is a column vector whose entries are scalar polynomials. Consequently,
\[
F(z)^*F(z)
 \,=\,
R(z)^*A(z)^*A(z)R(z).
\]
Writing
\[
A(z) \,=\,
\begin{pmatrix}
\alpha_1(z)\\
 \vdots\\
\alpha_m(z)
\end{pmatrix},
\]
we have
$ A(z)^*A(z) = \sum_{\ell=1}^m |\alpha_\ell(z)|^2.$
By the scalar-valued Fej\'er--Riesz factorization, there exists a scalar polynomial $h \in \mathbb C[z]$ such
that
\[
        \sum_{\ell=1}^m |\alpha_\ell(z)|^2 \,=\, |h(z)|^2,
        \qquad z\in\mathbb T.
\]
Therefore
\[
       P_{\psi} \,=\,  F(z)^*F(z)
        \,=\,
        |h(z)|^2 R(z)^*R(z).
\]
Since
$  Q(z)^*Q(z)   = |g_\varphi(z)|^2 R(z)^*R(z), $
the inequality
$ F(z)^*F(z) \preceq Q(z)^*Q(z) $
is equivalent to
\[
        0 \,\leq\, |h(z)|^2 \,\leq\,  |g_\varphi(z)|^2,
        \qquad z\in\mathbb T.
\]
Here we use that $R(z)\neq 0$ for every $z\in\mathbb T$, which follows from
$\gcd(r_1,\ldots,r_n)=1$.

\smallskip

Finally, since $\deg R=\vd_r$ and the entries of $F$ have degree at most
$\vd-1$, each polynomial $\alpha_\ell$ has degree at most $\vd-1-\vd_r$. Hence the
scalar Fej\'er--Riesz factor $h$ may also be chosen with
\[
\deg h \,\leq\, \vd-1-\vd_r.
\]
	\end{proof}
	
\subsection{Proof of Theorem \ref{thmA}}

\begin{proof}
Let $\varphi: \mathcal T_{\vd} \to M_{n}$ be a pure UCP map with polynomial density $P_{\varphi}.$ By the pure extension theorem, there exists a pure UCP map $\tilde{\varphi} : M_{\vd} \to M_{n}$ such that $\tilde{\varphi}|_{\mathcal T_{\vd}} = \varphi.$ By Lemma \ref{lem:pureucp}, there exists an isometry $V: \mathbb C^{n} \to \mathbb{C}^{\vd}$ such that 
\[
\varphi (T) \,=\, V^{*}TV, \qquad T \in \mathcal T_{\vd}. 
\]
This forces $1 \leq n \leq \vd.$ By Proposition \ref{prop:compression-density}, there exists a row polynomial 
\[  
Q(z) \,=\,\begin{pmatrix} q_1(z)&q_2(z)&\cdots&q_n(z)
	\end{pmatrix},
        \qquad q_j\in \mathbb C[z],\quad \deg q_j \leq \vd-1,
\]
such that $P_{\varphi} = Q(z)^{*}Q(z)$ for all $z \in \mathbb T.$ Write
\[
Q(z)\,=\,g_\varphi(z)R(z),
\qquad
R(z) \,=\, \begin{pmatrix} r_1(z)&r_2(z)&\cdots&r_n(z)\end{pmatrix},
\]
where $g_\varphi=\gcd(q_1,\ldots,q_n)$ is chosen so that
$\gcd(r_1,\ldots,r_n)=1$. Let $\vd_{r} = \max_{j} \deg r_{j}.$ Then $\deg g_{\varphi} \leq \vd-1- \vd_{r}.$  We claim that $|g_{\varphi}|^{2}$ spans an extreme ray of $\mathcal{P}^{+}_{(\vd -1 - \vd_{r})}(\mathbb T).$ Let $h\in \mathbb C[z]$ be polynomial of degree at most $\vd-1-\vd_{r}$ such that $|h(z)| \leq |g_{\varphi} (z)|$ for all $z \in \mathbb T.$ By Proposition \ref{prop:domination}, the CP map $\psi : \mathcal T_{\vd} \to M_{n}$ with polynomial density 
\[
P_{\psi}(z) \,=\, |h(z)|^{2} R(z)^{*}R(z), \qquad z \in \mathbb T,  
\]
satisfies $\psi \leq_{\rm cp} \varphi.$ Since $\varphi$ is pure, $\psi = t\, \varphi$ for some $t \in [0,1].$ Thus $|h(z)|^{2} = t\, |g_{\varphi} (z)|^{2}$ for all $z \in \mathbb T.$ This proves our claim. Lemma \ref{lem:extreme ray} now implies that
$ \deg g_{\varphi}=\vd-1-\vd_r,$
and every zero of $g_{\varphi}$ lies on $\mathbb T$. Since $ \deg g_{\varphi}=\vd-1-\vd_r,$ we get $\max_{j} \deg q_{j} = \vd-1.$

\smallskip

Conversely, assume that $\varphi: \mathcal T_{\vd} \to M_{n}$ is a UCP map with  polynomial density $P_{\varphi}$ which satisfies conditions \textup{(1)} and \textup{(2)}. Let $\psi : \mathcal T_{\vd} \to M_{n}$ be a CP map such that $  \psi \leq_{\rm cp} \varphi.$ By Proposition \ref{prop:domination}, there exists a scalar polynomial $h \in \mathbb C[z]$ with $\deg h \leq \vd-1-\vd_{r}$ such that 
\[
P_{\psi}(z) \,=\, |h(z)|^{2} R(z)^{*} R(z), \qquad z \in \mathbb T,
\]
and $|h(z)| \leq |g_{\varphi}(z)|$ for all $z \in \mathbb T.$ Since $\deg g_\varphi = \vd-1-\vd_{r}$ and all zeros of $g_{\varphi}$ lie on the circle,  by Lemma \ref{lem:extreme ray}, $|g_{\varphi}|^{2}$ spans an extreme ray of $\mathcal{P}^{+}_{\vd-1-\vd_{r}}(\mathbb T).$ Therefore, there exists $t \in [0,1]$ such that $|h(z)|^{2} = t\, |g_{\varphi}(z)|^{2}$ for all $z \in \mathbb T.$ This shows that $\psi = t\, \varphi.$ Hence $\varphi$ is pure. 
\end{proof}

In the following lemma, we show that the row polynomial $Q$ appearing in Theorem \ref{thmA} is unique up to multiplication by a unimodular constant.

\begin{lemma}\label{lem:row-factor-unique}
Let $\varphi: \mathcal T_{\vd} \to M_{n}$ be a UCP map. Let  
$ Q(z)=\begin{pmatrix}q_1(z)&\cdots&q_n(z)\end{pmatrix}$
 be a row polynomial satisfying conditions \textup{(1)} and \textup{(2)} of Theorem \ref{thmA}. 
If
$ R(z)=\begin{pmatrix}r_1(z)&\cdots&r_n(z)\end{pmatrix} $
is a row polynomial of degree at most \(\vd-1\) such that
$ P_{\varphi}(z) = R(z)^*R(z) $ for all  $ z\in\mathbb T$,
then there exists \(\lambda\in\mathbb C\) with \(|\lambda|=1\) such that
\[
R(z) \,=\, \lambda Q(z), \qquad z \in \mathbb T.
\]
In particular, if $R$ is a row polynomial with $P_{\varphi} = R^{*}R,$ then $R$ satisfies conditions \textup{(1)} and \textup{(2)} of Theorem \ref{thmA}.
\end{lemma}

\begin{proof}
Write
\[
Q(z) \,=\, g(z) Q_0(z),
\qquad
Q_0(z) \,=\, \begin{pmatrix}q_1^0(z)&\cdots&q_n^0(z)\end{pmatrix},
\]
where $\gcd(q_1^0,\ldots,q_n^0)=1$. Let
$ \vd_{0}=\max_j \deg q_j^0.$
Since $Q$ satisfies the full-degree condition, we have
$ \deg g+\vd_0=\vd-1.$
Since $P_{\varphi}=R(z)^*R(z)$, the proof of Proposition \ref{prop:domination} shows that there exists an analytic polynomial $h$ of degree at most $\vd-1-\vd_0$ such that
\[
R(z) \,=\, h(z)Q_0(z),\qquad z\in\mathbb T.
\]
Now 
\[
P_{\varphi}(z) \,=\, Q(z)^*Q(z) \,=\, R(z)^*R(z)
\]
implies
\[
|g(z)|^2 Q_0(z)^*Q_0(z) \,=\, |h(z)|^2 Q_0(z)^*Q_0(z),\qquad z\in\mathbb T.
\]
Since $\gcd(q_1^0,\ldots,q_n^0)=1$, the row $Q_0(z)$ is not identically zero at any point of $\mathbb T$. Hence
\[
|h(z)| \,=\, |g(z)|,\qquad z\in\mathbb T.
\]
Since all zeros of $g$ lie on $\mathbb T$, by an argument similar to the proof of Lemma \ref{lem:extreme ray}, it follows that $h=\lambda g$ for some unimodular constant $\lambda$. This completes the proof.
\end{proof}

\subsection{Checkable criterion for purity}\label{ssec:checkable condition}
Theorem \ref{thmA} gives a characterization of pure UCP maps. In particular, it immediately provides abundant examples of such maps. However, so far it is less clear whether this characterization gives a checkable criterion: given a UCP map $\varphi:\mathcal T_{\vd}\to M_n$, can one decide whether $\varphi$ is pure? At first glance, this seems nontrivial. Indeed, to show that $\varphi$ is not pure, one would have to rule out every factorization of $P_\varphi$ satisfying conditions \textup{(1)} and \textup{(2)}.

\smallskip

We shall show that this difficulty is only apparent: any Fej\'er--Riesz
factorization of $P_\varphi$ contains enough information to decide whether
$\varphi$ is pure.

\smallskip

Let 
\[
P_{\varphi}(z) = F(z)^{*} F(z), \qquad z \in \mathbb T,
\]
be any Fej\'er--Riesz factorization of $P_{\varphi}.$ So $F$ is an $m \times n$ matrix-valued analytic polynomial. Write
\[
F(z)=
\begin{pmatrix}
f_1(z)\\
\vdots\\
f_m(z)
\end{pmatrix},
\]
where each \(f_\ell\) is a
\(1\times n\) row polynomial. If $\varphi$ is pure, then using the same argument as in the proof of Proposition \ref{prop:domination}, we get that all rows of $F$ are scalar multiples of a single nonzero row polynomial. Thus, we obtain a first checkable obstruction:
if the rows of $F$ are not all scalar multiples of one nonzero row polynomial,
then $\varphi$ is not pure.

\smallskip

Now choose a nonzero
row, say $ R(z)=f_p(z).$
Then for each \(\ell\) there is a scalar \(c_\ell\in\mathbb C\) such that
$ f_\ell(z)=c_\ell R(z).$
Thus
\[
F(z) \,=\, cR(z),
\qquad
c=
\begin{pmatrix}
c_1\\
\vdots\\
c_m
\end{pmatrix}.
\]
Let 
\[
\alpha=\|c\|,
\qquad
h=\frac{c}{\alpha},
\qquad
Q(z)=\alpha R(z).
\]
Then \(\|h\|=1\) and $ F(z)=hQ(z).$
Consequently
\[
P_\varphi(z)=F(z)^*F(z)=Q(z)^*Q(z).
\]
Now $Q$ is a row polynomial with $ P_{\varphi}=Q^*Q.$
If $\varphi$ is pure, then Lemma \ref{lem:row-factor-unique} shows that
$Q$ must satisfy conditions \textup{(1)} and \textup{(2)} of Theorem \ref{thmA}.
Thus the purity of $\varphi$ can be checked directly from $Q$: the map
$\varphi$ is pure if and only if
$ \max_j \deg q_j=\vd-1 $
and $\gcd(q_1,\ldots,q_n)$ has all its zeros on $\mathbb T$.

\subsection{Purity in terms of isometry} \label{ssec:isometry}
A pure UCP map $\varphi:\mathcal T_{\vd}\to M_n$ must be of the form
\[
\varphi(T)=V^*TV,\qquad T\in\mathcal T_{\vd},
\]
for some isometry $V:\mathbb C^n\to\mathbb C^\vd$. However, not every UCP map of this form is pure. To decide when such a map is pure, one needs a Fej\'er--Riesz factorization of $P_\varphi$. In the following lemma, we associate to the isometry $V$ a row polynomial $Q_V$ such that
\[
P_\varphi(z) \,=\, Q_V^* (z) Q_V(z), \qquad z \in \mathbb T .
\]
Thus, the purity of $\varphi$ can then be decided directly from this row polynomial.

\begin{lemma}\label{lem:compression-density-factor}
Let
$ \varphi:\mathcal T_{\vd}\to M_n $
be given by
$ \varphi(T) = V^*TV,$ 
where \(V:\mathbb C^n\to\mathbb C^{\vd}\) is an isometry. Write
\[
V \,=\,
\begin{pmatrix}
V_0\\
V_1\\
\vdots\\
V_{\vd-1}
\end{pmatrix},
\qquad
V_i\in M_{1,n}.
\]
Define the row polynomial
$ Q(z) \,=\, \sum_{i=0}^{\vd-1}V_i z^i. $
Then the polynomial density of \(\varphi\) is
$ P_\varphi(z)=Q(z)^*Q(z). $

\end{lemma}

\begin{proof}
For \(0\leq k\leq \vd-1\), we have
$ J^k=\sum_{i=0}^{\vd-1-k}E_{i,i+k}.$
Therefore
\[
\varphi(J^k)
\,=\,
V^*J^kV
\,=\,
\sum_{i=0}^{\vd-1-k}V_i^*V_{i+k}.
\]
On the other hand,
\[
Q(z)^*Q(z)
\,=\,
\left(\sum_{i=0}^{\vd-1}V_i z^i\right)^*
\left(\sum_{j=0}^{\vd-1}V_j z^j\right)
\,=\,
\sum_{i,j=0}^{\vd-1}V_i^*V_j z^{j-i}.
\]
Hence the coefficient of \(z^k\) in \(Q(z)^*Q(z)\) is
\[
\sum_{i=0}^{\vd-1-k}V_i^*V_{i+k}
=
\varphi(J^k),
\]
and the coefficient of \(z^{-k}\) is \(\varphi(J^k)^*\). This completes the proof. 
\end{proof}

\begin{remark}
An interested reader may compare our scalar case $(n=1)$ with \cite[Proposition 2.12]{Hekkelman22} and \cite[Corollary 2.13]{Hekkelman22}. For a unit vector $\xi=(\xi_0,\ldots,\xi_{\vd-1})^{t}$, the vector state $T \mapsto \langle T \xi , \xi \rangle$ leads to two different polynomial conventions. Hekkelman's polynomial is
\[
Q_\xi(z) \,=\, \sum_{k=0}^{\vd-1}\xi_k z^{\vd-k-1},
\]
whereas Lemma \ref{lem:compression-density-factor} gives
\[
\widetilde{Q}_\xi(z) \,=\, \sum_{k=0}^{\vd-1}\xi_k z^k.
\]
The condition that $Q_\xi$ has full degree and all its zeros on $\mathbb T$ is equivalent to the same condition for $\widetilde Q_{\xi}$, and hence \cite[Proposition 2.12]{Hekkelman22} agrees with our scalar characterization. However, the density of the vector state is $|\widetilde Q_\xi|^2$, not in general $|Q_\xi|^2$; for instance, when $\vd=2$, these are $|\xi_0+\xi_1z|^2$ and $|\xi_0z+\xi_1|^2$, which need not be equal.
\end{remark}

 \section{Unique Completely Positive Extension} \label{sec:unique-extension}

Let $\varphi:\mathcal T_{\vd}\to M_n$ be a CP map, and let
$P_\varphi$ denote its polynomial density. We characterize unique CP
extension of $\varphi$ to $M_{\vd}$ in terms of the Fej\'er--Riesz factors of
$P_\varphi$.

\smallskip

Let 
$ F(z) = \sum_{j=0}^{\vd-1}F_jz^j,$ $F_j\in M_{m,n} $ be a Fej\'er--Riesz factor of $P_\varphi.$ To such a factor, we associate its coefficient Gram matrix
\[
G_F \,=\, [F_i^*F_j]_{i,j=0}^{\vd-1} \,\in\, M_{\vd}(M_n).
\]
Since the coefficient of $z^k$ in $F(z)^*F(z)$ is
$ \sum_{j=0}^{\vd-1-k}F_j^*F_{j+k}, $
the identity $P_\varphi=F^*F$ gives
\[
\sum_{j=0}^{\vd-1-k}F_j^*F_{j+k}
\,=\,
\widehat P_\varphi(k)
\,=\,
\varphi(J^k),
\qquad 0\leq k\leq \vd-1.
\]

Recall that a CP extension of $\varphi$ is a CP map
$\tilde\varphi:M_{\vd}\to M_n$ satisfying
$\tilde\varphi|_{\mathcal T_{\vd}}=\varphi$. Its Choi matrix is
\[
C_{\tilde\varphi}
=
[\tilde\varphi(E_{i,j})]_{i,j=0}^{\vd-1}
\in M_{\vd}(M_n),
\]
and by Choi's theorem $\tilde\varphi$ is CP if and only if
$C_{\tilde\varphi}\succeq 0$. Since
$ J^k=\sum_{i=0}^{\vd-1-k}E_{i,i+k}, $
the condition that $\tilde\varphi$ extend $\varphi$ is equivalent to
\[
\sum_{i=0}^{\vd-1-k}
\tilde\varphi(E_{i,i+k})
\,=\,
\varphi(J^k)
\,=\,
\widehat P_\varphi(k),
\qquad 0\leq k\leq \vd-1.
\]
Thus CP extensions of $\varphi$ are precisely positive block matrices
$C=[C_{i,j}]_{i,j=0}^{\vd-1}\in M_{\vd}(M_n)$ satisfying
\[
\sum_{i=0}^{\vd-1-k}C_{i,i+k}
=
\widehat P_\varphi(k),
\qquad 0\leq k\leq \vd-1.
\]

\begin{proposition} \label{prop:UCP}
Let $\varphi:\mathcal T_{\vd}\to M_n$ be CP. Then $\varphi$ has a unique CP
extension to $M_{\vd}$ if and only if all Fej\'er--Riesz factors of
$P_\varphi$ have the same coefficient Gram matrix. Equivalently, whenever
\[
P_\varphi(z) \,=\, F(z)^*F(z) \,=\, H(z)^*H(z),
\]
with
$ F(z)=\sum_{j=0}^{\vd-1}F_jz^j$ 
and 
$ H(z)=\sum_{j=0}^{\vd-1}H_jz^j$,
one has
\[
F_i^*F_j
\,=\,
H_i^*H_j, \qquad 0 \leq i,j \leq \vd-1.
\]
\end{proposition}

\begin{proof}
First let $F(z)=\sum_{j=0}^{\vd-1}F_jz^j$ be a Fej\'er--Riesz factor of
$P_\varphi$. Define $\Psi_F:M_{\vd}\to M_n$ on matrix units by
\[
\Psi_F(E_{i,j}) \,=\, F_i^*F_j,
\qquad 0\leq i,j\leq \vd-1.
\]
Its Choi matrix is $G_F=[F_i^*F_j]_{i,j=0}^{\vd-1}$, which is positive
because it is a Gram matrix. Hence $\Psi_F$ is CP. Moreover, for
$0\leq k\leq \vd-1$,
\[
\Psi_F(J^k)
\,=\,
\sum_{i=0}^{\vd-1-k}F_i^*F_{i+k}
\,=\,
\widehat P_\varphi(k)
\,=\,
\varphi(J^k),
\]
and similarly $\Psi_F((J^*)^k)=\varphi((J^*)^k)$. Since
$\mathcal T_{\vd}$ is spanned by
\[
I,J,\ldots,J^{\vd-1},J^*,\ldots,(J^*)^{\vd-1},
\]
we have $\Psi_F|_{\mathcal T_{\vd}}=\varphi$. Thus the Fej\'er--Riesz
factor $F$ of $P_\varphi$ gives a CP extension of $\varphi$, whose Choi matrix is
$G_F$.

\smallskip

Conversely, let $\tilde\varphi:M_{\vd}\to M_n$ be a CP extension of
$\varphi$, and write
$ C_{\tilde\varphi}
=
[\tilde\varphi(E_{i,j})]_{i,j=0}^{\vd-1}. $
Since $C_{\tilde\varphi}\succeq 0$, it admits a Gram factorization
\[
\tilde\varphi(E_{i,j}) \,=\, F_i^*F_j,
\qquad 0\leq i,j\leq \vd-1,
\]
for some matrices $F_0,\ldots,F_{\vd-1}$ of a common size. Setting
$ F(z)=\sum_{j=0}^{\vd-1}F_jz^j, $
we obtain, for $0\leq k\leq \vd-1$,
\[
\sum_{j=0}^{\vd-1-k}F_j^*F_{j+k}
\,=\,
\sum_{j=0}^{\vd-1-k}\tilde\varphi(E_{j,j+k})
\,=\,
\tilde\varphi(J^k)
\,=\,
\varphi(J^k)
\,=\,
\widehat P_\varphi(k).
\]
The same argument applied to $(J^*)^k$ gives agreement of the negative
Fourier coefficients. Hence $F(z)^*F(z)=P_\varphi(z)$ on $\mathbb T$, so
$F$ is a Fej\'er--Riesz factor of $P_\varphi$, and
$ C_{\tilde\varphi}=G_F. $

\smallskip

Therefore CP extensions of $\varphi$ correspond exactly to the coefficient
Gram matrices of Fej\'er--Riesz factors of $P_\varphi$. It follows
immediately that $\varphi$ has a unique CP extension if and only if all such
factors have the same coefficient Gram matrix.
\end{proof}

\begin{remark}
The criterion is purely polynomial. The Toeplitz data of $\varphi$ determine
only the diagonal sums
\[
\sum_{i=0}^{\vd-1-k}F_i^*F_{i+k}
=
\varphi(J^k),
\qquad 0\leq k\leq \vd-1,
\]
whereas a CP extension to $M_{\vd}$ is determined by the full coefficient
Gram matrix $[F_i^*F_j]_{i,j=0}^{\vd-1}$. Thus uniqueness of CP extension is
equivalent to the rigidity of this full Gram matrix among all
Fej\'er--Riesz factors of $P_\varphi$.
\end{remark}

\subsection{Proof of Theorem \ref{thmB}}

\begin{proof}
Let $F$ be a Fej\'er--Riesz factor of $P_{\varphi},$ that is, $F$ is a matrix valued analytic polynomial of degree at most $\vd-1$ satisfying 
$ F(z)^*F(z) = P_\varphi(z).$

\smallskip

Since $\varphi$ is pure, by Theorem \ref{thmA}, there exists a row polynomial $Q $ with 
$ P_\varphi(z) = Q(z)^*Q(z),$
satisfying conditions \textup{(1)} and \textup{(2)} of Theorem \ref{thmA}. It follows from the discussion in Subsection \ref{ssec:checkable condition} that there exists a constant unit column vector $h_{F}$ such that $F = h_{F}Q. $ 

\smallskip

Write
$ Q(z)=\sum_{i=0}^{\vd-1}Q_i z^i.$
Then
$ F_i = h_{F} Q_i,$ for $ 0\leq i\leq \vd-1.$
Since $h_{F}^{*}h_{F} = 1$, we get
\[
F_i^*F_j
\,=\,
Q_i^*h_{F}^* h_{F} Q_j
\,=\,
Q_i^*Q_j
\qquad
0\leq i,j\leq \vd-1.
\]
It follows from Proposition \ref{prop:UCP} that $\varphi$ has a unique CP extension to $M_{\vd}.$
\end{proof}

\begin{corollary} \label{cor:uniqueness of isometry}
Let \(V,W:\mathbb C^n\to\mathbb C^{\vd}\) be isometries, and suppose that
\[
V^*TV=W^*TW,\qquad T\in\mathcal T_{\vd}.
\]
Assume that the UCP map
$ \varphi(T)=V^*TV $
is pure. Then there exists \(\lambda\in\mathbb T\) such that
$ W=\lambda V. $
\end{corollary}

\begin{proof}
Since \(\varphi\) is pure, it has a unique UCP extension to
\(M_{\vd}=C^*(\mathcal T_{\vd})\). The maps
\[
A\mapsto V^*AV
\qquad\text{and}\qquad
A\mapsto W^*AW
\]
are both UCP extensions of \(\varphi\) to \(M_{\vd}\). By uniqueness, they
are equal:
\[
V^*AV=W^*AW,\qquad A\in M_{\vd}.
\]

Now we use the uniqueness of minimal Stinespring dilations. Both maps are obtained
by compressing the identity representation of \(M_{\vd}\) on
\(\mathbb C^{\vd}\). This Stinespring representation is minimal. Hence there is a unitary
$ U:\mathbb C^{\vd}\to\mathbb C^{\vd} $
such that $ UV=W $
and $UA=AU$ for all $ A\in M_{\vd}.$
But the commutant of \(M_{\vd}\) is just the scalars. 
Therefore
$ U=\lambda I_{\vd} $
for some $ \lambda\in\mathbb T $. Hence
$ W=UV=\lambda V.$
\end{proof}

\begin{remark}\label{rem:comparison}
Theorem \ref{thmB} should be compared with the usual unique extension property in noncommutative Choquet theory. In Arveson's framework \cite{Arveson69, Arveson72, Arveson08}, and in the subsequent work of Dritschel and McCullough \cite{DM05}, maximal UCP maps are characterized by the unique extension property, where the unique extension to the generated $C^*$-algebra is a $*$-representation. This point of view plays a central role in the theory of boundary representations and the $C^*$-envelope, see also \cite{DK15, DK24, Kleski14}. Our result is of a different nature. The pure UCP maps considered here need not be restrictions of representations; rather, they are typically compressions $ T\mapsto V^*TV$. Nevertheless, in the Toeplitz setting, purity forces a unique UCP extension to $ M_\vd$, although this extension is generally not multiplicative unless the compression is trivial. Thus, the finite Toeplitz operator system exhibits a form of extension rigidity for pure matrix states which is weaker than maximality in Arveson's sense, but stronger than what holds for general hyperrigid operator systems. 
\end{remark}

\subsection{An example and a counter-example}
First, we give an example of a UCP map $\varphi: \mathcal T_{2} \to M_{2}$ which is not pure but has a unique UCP extension. 
\begin{example}
Let
\[
\xi \,=\, \frac1{\sqrt2}\begin{pmatrix}1\\1\end{pmatrix},
\qquad
\rho(T)=\langle T\xi,\xi\rangle,\quad T\in\mathcal T_2.
\]
Define
$ \varphi:\mathcal T_2\to M_2,$ by 
$ \varphi(T)=\rho(T)I_2.$
Then $ \varphi $ is a UCP map with a unique UCP extension, but it is not pure.
\end{example}
\begin{proof}
First, $ \rho $ is pure by \cite[Proposition 2.12]{Hekkelman22}. Now let
$ \tilde\varphi:M_2\to M_2 $
be a UCP extension of $ \varphi$. Its Choi matrix has the form
\[
C_{\tilde\varphi} \,=\,
\begin{bmatrix}
D&\frac12 I_2\\
\frac12 I_2&I_2-D
\end{bmatrix}
\succeq 0,
\]
where $ D=\tilde\varphi (E_{11}) $. For every $ x\in\mathbb C^2 $,
\[
\Big\langle
C_{\tilde\varphi}
\begin{pmatrix}x\\-x\end{pmatrix},
\begin{pmatrix}x\\-x\end{pmatrix}
\Big\rangle
\,=\,
\left\langle Dx,x \right\rangle+ \left\langle (I_2-D)x,x \right\rangle
- \left\langle I_2x,x \right\rangle
\,=\,0.
\]
Since $ C_{\tilde\varphi} \succeq 0 $, this implies
$ C_{\tilde\varphi}
\begin{pmatrix}x\\-x\end{pmatrix}=0. $
Looking at the first component gives
$ Dx-\frac12x=0 $ for every $ x $. Hence
$ D=\frac12 I_2. $
Therefore
$ C_{\tilde\varphi} =
\frac12
\begin{bmatrix}
I_2&I_2\\
I_2&I_2
\end{bmatrix}, $
so the Choi matrix of any UCP extension is uniquely determined. Hence
$ \varphi $ has a unique UCP extension.

\smallskip

Finally, $ \varphi $ is not pure. Let
$ P=
\begin{bmatrix}
1&0\\
0&0
\end{bmatrix}. $
Define $ \psi:\mathcal T_2\to M_2$ by  $ \psi(T)=\rho(T)P.$
Then $ \psi $ is CP and
$ \varphi-\psi=\rho(\cdot)(I_2-P) $ is also CP. Thus
$  \psi\leq_{\rm cp} \varphi.$ But $ \psi $ is not a scalar multiple of $ \varphi$, since $ \psi(I)=P $ is not a scalar multiple of $ \varphi(I)=I_2.$
Therefore, $ \varphi $ is not pure.
\end{proof}

Next, we give an example of a finite-dimensional hyperrigid operator system $\mathcal S \subseteq M_4$ and a pure state $\varphi:\mathcal{S} \to \mathbb{C}$ which does not have a unique UCP extension. Thus, the hyperrigidity of $ \mathcal S$ does not imply that every pure UCP map on $ \mathcal S $ has a unique UCP extension.

\begin{example} \label{ex: counter example}

Let
$ C=\frac{1}{\sqrt2}
	\begin{bmatrix}
		1&1\\
		1&-1
	\end{bmatrix}.$
Define two unitaries $ U,W\in M_4 $ by
$ U =
\begin{bmatrix}
	0&I_2\\
C&0
\end{bmatrix} $  and 
 $ W = 
\operatorname{diag}(1,1,-1,i). $
Let
\[
\mathcal S \,=\, \text{span}\{I,U,U^*,W,W^*\} \,\subseteq\, M_4.
\]
Then $\mathcal{S}$ is a hyperrigid operator system, but the vector state
\[
\varphi:\mathcal{S} \to \mathbb{C},
\qquad
\varphi(a) \,=\, \langle a e_1,e_1\rangle,
\]
is pure and does not have a unique UCP extension.

\end{example}

\begin{proof}
The operator system is hyperrigid: First, we show that
$ C^*(\mathcal S) = C^{*}(U,W) = M_4.$ Let $X\in M_4(\mathbb{C})$ commute with both $U$ and $W$. Since $ W=\operatorname{diag}(1,1,-1,i),$ the eigenspaces of $ W $ are $ {\rm span}\{e_1,e_2\},\, \mathbb{C} e_3,\, \mathbb{C} e_4.$ Therefore $ XW = WX$ implies that $X$ has the block form
$ X= \begin{bmatrix}
		A&0\\
		0&B
	\end{bmatrix},$
where $A\in M_2(\mathbb{C})$ and $B=
	\begin{bmatrix}
		b&0\\
		0&c
	\end{bmatrix}.$
Now imposing $ XU = UX $ we get 
$ A=B $ and $ BC = CA.$
Since $ A=B $, this becomes
$ BC = CB.$ But
$ B= \begin{bmatrix}
		b&0\\
		0&c
	\end{bmatrix} $
commutes with
$ C=\frac1{\sqrt2}
	\begin{bmatrix}
		1&1\\
		1&-1
	\end{bmatrix} $
if and only if $b=c$. Therefore $B=bI_2$ and hence $A=B=bI_2$. Thus $X=bI_4$.

\smallskip
    
So the commutant of $C^*(U, W)$ is only the scalars. Since $C^*(U,W)$
is a finite-dimensional unital $^*$-subalgebra of $M_4$, it
follows that
$ C^*(U,W)=M_4(\mathbb{C}).$

\smallskip

Now we prove the hyperrigidity. Let
$ \pi:M_4(\mathbb{C})\to \mathcal B(\mathcal H) $
be a representation, and let
$ \Psi:M_4 \to \mathcal B(\mathcal H) $
be UCP such that
$ \Psi|_\mathcal{S}=\pi|_\mathcal{S}.$
Since $U,W\in \mathcal{S}$, we have
\[
\Psi(U) \,=\,\pi(U),
\qquad
\Psi(W) \,=\, \pi(W).
\]
The operators $\pi(U)$ and $\pi(W)$ are unitaries. Hence
$ \Psi(U)^*\Psi(U)=I=\Psi(U^*U),$ and
$ \Psi(U)\Psi(U)^*=I=\Psi(UU^*). $
Therefore $U$ lies in the multiplicative domain of $\Psi$. Similarly, $W$ lies in the multiplicative domain of $\Psi$.

\smallskip

Since $U$ and $W$ generate $M_4(\mathbb{C})$, the multiplicative domain of $\Psi$ contains all of $M_4(\mathbb{C})$. Thus $\Psi$ is a $^*$-homomorphism on $M_4$. Because it agrees with $\pi$ on the generators $U$ and $W$, it follows that
$ \Psi=\pi.$
Hence, every representation of $M_4 $ has the unique extension property relative to $\mathcal{S}$. Thus $\mathcal{S}$ is hyperrigid.

\smallskip

$\varphi$ is a pure state on $\mathcal{S}$:
Let
\[
h \,=\, {\rm Re}\, W=\frac{W+W^*}{2}
	=\operatorname{diag}(1,1,-1,0)\,\,
    \preceq I.
	\]
Moreover $\varphi(h)=1.$ Consider the exposed face of the state space of $\mathcal{S}$
\[
F \,=\, \{\omega\in {\rm UCP}(\mathcal{S}, \mathbb C) \ :\  \omega(h)=1\}.
	\]
We claim that
$ F=\{\varphi\}.$
Let $\omega\in F$. Extend $\omega$ to a state
$ \tilde\omega:M_4(\mathbb{C})\to\mathbb{C}.$
Since
$ I-h = \operatorname{diag}(0,0,2,1)\succeq 0 $ and
$ \tilde\omega(I-h)=1-\omega(h)=0,$
the state $\tilde\omega$ is supported on the kernel of $I-h$, namely
on the subspace ${\rm span}\{e_1,e_2\}.$ Equivalently, if $P$ denotes the projection onto ${\rm span}\{e_1,e_2\}$, then
$ \tilde\omega(a)=\tilde\omega(PaP),$ $ a\in M_4.$
Now
$ PWP=P $ and $ PUP=0.$ Therefore
$\omega(W)=1, \omega(U)=0. $
	Also
	$
	\omega(W^*)=1$ and
	$
	\omega(U^*)=0.
	$
	Since $S$ is spanned by $I,U,U^*,W,W^*$, these values determine
	$\omega$ uniquely. They are exactly the values of $\varphi$. Hence
	$ \omega=\varphi. $
	Thus
	$ F=\{\varphi\}.$
	Since $\{\varphi\}$ is an exposed face of the state space of $S$,
	$\varphi$ is an extreme point of the state space. Therefore
	$\varphi:\mathcal{S} \to \mathbb{C}$ is a pure state.

    \smallskip
    
    The state $\varphi$ does not have unique CP extension:
Define two states on $M_4(\mathbb{C})$ by
\[
\tilde\varphi_1(a) \,=\, \langle ae_1,e_1\rangle,
\qquad
\tilde\varphi_2(a) \,=\, \langle ae_2,e_2\rangle.
\]
These are distinct states on $M_4$. However, they agree on $\mathcal{S}$.

\end{proof}

\noindent \textbf{Acknowledgments.}
The first author thanks her research supervisor, Prof. Tirthankar Bhattacharyya, for useful discussions. The second author thanks Dr. Poornendu Kumar for suggestions that helped improve the introduction of the paper. Prof. Walter van Suijlekom's visit to India under the Jubilee Chair Professorship of the Indian Academy of Sciences has been influential in our interest in the Toeplitz operator system.

\end{document}